\documentclass[11pt,reqno,letterpaper]{amsart}
\usepackage[all]{xy}
\usepackage{hyperref}
\usepackage{amsfonts,mathrsfs,bbm,latexsym,rawfonts,amsmath,amssymb,amsthm,epsfig}
\usepackage{setspace, color}
\usepackage{cases}
\usepackage{todonotes}
\newtheorem{thm}{Theorem}[section]
\newtheorem{cor}[thm]{Corollary}
\newtheorem{rem}[thm]{Remark}
\newtheorem{lem}[thm]{Lemma}
\newtheorem{prop}[thm]{Proposition}

\numberwithin{equation}{section}

\newcommand{\al}{\alpha}

\newcommand{\ld}{\lambda}

\newcommand{\de}{\delta}
\newcommand{\De}{\Delta}
\newcommand{\ep}{\varepsilon}

\newcommand{\Om}{\Omega}
\newcommand{\ga}{\gamma}
\newcommand{\Ga}{\Gamma}
\newcommand{\ka}{\kappa}
\renewcommand{\th}{\theta}

\newcommand{\F}{\mathcal{F}}

\renewcommand{\P}{\mathcal{P}}

\newcommand{\g}{\mathfrak{g}}

\newcommand{\C}{\mathscr{C}}

\newcommand{\Real}{\mathbb{R}}

\newcommand{\norm}[1]{\Vert#1\Vert}

\def\<{\left\langle} \def\>{\right\rangle}
\def\({\left(} \def\){\right)}
\newcommand{\n}{\nabla}
\newcommand{\p}{\partial}

\begin{document}
\title[]{Removable singularities for Yang-Mills-Higgs fields in higher dimensions}
\author[B. Chen]{Bo Chen}
\address{School of Mathematics, South China University of Technology, Guangzhou 510640, People's Republic of China}
\email{cbmath@scut.edu.cn}



\begin{abstract}
This paper establishes decay estimates near isolated singularities for $n$-dimensional Yang-Mills-Higgs fields defined on a fiber bundle ($n \geq 4$). These estimates yield a removable singularity theorem for Yang-Mills-Higgs fields under conformally invariant energy bounds, extending the classical results for Yang-Mills fields and harmonic maps.
\end{abstract}
\maketitle

\section{Introduction}

This paper is a sequel to \cite{CH1}, in which decay estimates were derived for three-dimensional Yang-Mills-Higgs fields under appropriate energy constraints. Here, we extend the analysis to higher dimensions and derive decay estimates near isolated singularities for $n$-dimensional Yang-Mills-Higgs fields defined on a fiber bundle, where $n\geq 4$. The curvature of the base manifold plays no essential role in the local regularity theory of Yang-Mills-Higgs fields. We therefore simplify the setting by assuming that the base manifold is flat.

\subsection{Yang-Mills-Higgs fields on a ball}

Let $\Real^n$ denote $n$-dimensional Euclidean space, and let $B_1\subset \Real^n$ be the unit ball. Let $(N,h)$ be a compact Riemannian manifold equipped with an effective action of a compact connected Lie group $G$ that preserves the metric $h$. Let $\P$ be a trivial principal $G$-bundle over $B_1$, and let $\F=\P\times_G N$ be the associated bundle with fiber $N$. After fixing a trivialization, any section of $\F$ can be identified with a map $u:B_1\to N$, while a connection $A$ on $\P$ can be identified with a $\g$-valued 1-form, where $\g$ is the Lie algebra of $G$. For any pair $(A,u)$, the Yang-Mills-Higgs (YMH) functional is defined by
\begin{equation}\label{YMHf}
	\mathscr{E}(A, u)=\int_{B_1}(|F_A|^2+|\n_A u|^2+V(u))dx.
\end{equation}
Here $F_A=dA+\frac{1}{2}[A,A]$ is the curvature of the connection $A$, and $\n_A=\n+A$ is the covariant derivative induced by $A$, where $\n$ is the pullback connection on $u^*(TN)$. The function $V(u)$ is a gauge-invariant potential, referred to as the generalized Higgs potential on the bundle $\F$, and $dx$ is the volume element on $B_1$.

A YMH field $(A,u)$ is a critical point of the YMH functional and satisfies the following elliptic system:
\begin{equation}\label{eq-ymh0}
	\begin{cases}
		D^*_A F_A=-u^* \n_Au, \\[1ex]
		\n_A^* \n_Au=-2\n V(u),
	\end{cases}
\end{equation}
where $D^*_A$ is the adjoint of the exterior covariant derivative $D_A$. The term $u^*\n_A u$ lies in the dual space of $\Om^1(B_1,\g)$; namely, for every $c\in \Om^1(B_1,\g)$, we have
\[ \<u^*\n_Au, c\>=\<\n_Au, c\cdot u\>,\]
where $\cdot$ denotes the infinitesimal action of $\g$ on $N$.

\medskip
\subsection{Background and main results}
Originating in the study of electromagnetic phenomena, YMH fields are fundamental objects in classical gauge theory and particle physics. Mathematically, the YMH field equations form an elliptic system coupling the Yang-Mills equation with the harmonic map equation. Thus, a YMH field provides a natural gauge-theoretic generalization of a harmonic map. In the setting of symplectic fibrations, minimizing YMH fields are called symplectic vortices, and their moduli spaces can be used to define invariants of symplectic manifolds with Hamiltonian actions. See \cite{RI99,Mundet-Tian,S16} for further details.

Singularities of YMH fields have been studied extensively for several decades. Uhlenbeck \cite{U82} proved that a four-dimensional Yang-Mills field with a point singularity and finite Yang-Mills energy is gauge equivalent to a smooth connection. See also \cite{Ra93} for two alternative proofs of the removable singularity theorem for four-dimensional Yang-Mills fields. Parker \cite{Parker1} later extended Uhlenbeck's result to finite-energy Yang-Mills fields coupled to scalar and spinor fields. In dimension three and in dimensions greater than four, removable singularity theorems for both pure Yang-Mills (YM) fields and YM fields coupled to vector-bundle-valued Higgs fields have been studied in \cite{S84,SS84,S85,Smith,OSib,TT04,SU19}.

\medskip
In general, when the fiber is a nonflat Riemannian manifold, the nonlinear term in the Higgs-field equation substantially complicates the analysis of the asymptotic behavior of YMH fields near singularities. Consequently, even isolated removable singularities are difficult to study. The theory of isolated singularities of harmonic maps~\cite{SU81,L85} illustrates the essential differences between YMH fields on vector bundles and those on fiber bundles with curved fibers. In previous work \cite{CS21}, the present author and Song established a precise decay estimate for YMH fields with point singularities on a surface, leading to a removable singularity theorem for two-dimensional YMH fields with trivial limiting holonomy. More recently, in \cite{CH1}, we proved removability for three-dimensional YMH fields under suitable conformally invariant energy bounds. In dimensions greater than four, however, the decay behavior of YMH fields near singular points remains unclear.

\medskip
In this paper, we investigate decay estimates near isolated singularities for $n$-dimensional YMH fields satisfying the conformally invariant energy bound \eqref{inq3'}, where $n\geq 4$.

Throughout the paper, $C$ denotes a constant depending only on the dimension $n$ and the geometry of $N$. For $0<R\leq 1$, let $B_R=\{x\in\Real^n:\ |x|<R\}$ be the ball of radius $R$, and let $B_R^*=\{x\in\Real^n:\ 0<|x|<R\}$ be the corresponding punctured ball.

Our main results are as follows.
\begin{thm}\label{remov-singu}
There exists $\ep_0>0$ such that if $(A,u)$ is a smooth YMH field on $B^*_1\subset \Real^n$, where $n\geq 4$, satisfying
\begin{equation}\label{inq3'}
	\sup_{B_{\rho(y)}\subset B_{R_0}}\frac{1}{\rho^{n-2}}\int_{B_{\rho}(y)}\(|\n_A u|^2+|F_A|\)dx\leq \ep^2_0,
\end{equation}		
for some $R_0\leq 1$, then the singularity at the origin is removable. 
\end{thm}

A direct corollary of Theorem \ref{remov-singu} is as follows.
\begin{cor}\label{remov-singu1}
	Let $(A,u)$ be a smooth YMH field on $B^*_1\subset \Real^n$, where $n\geq 4$, satisfying the conformally invariant energy bound
	\begin{equation}\label{cor-con-enery}
		\int_{B_{1}}|F_A|^{\frac{n}{2}}dx+\int_{B_{1}}|\n_A u|^ndx<\infty.   
	\end{equation}
Then the singularity at the origin is removable.  
\end{cor}

Theorem \ref{remov-singu} is a consequence of the following result, which provides sharp decay estimates for YMH fields near isolated singularities.

\begin{thm}\label{main-thm}
There exists $\ep_0>0$ such that if $(A,u)$ is a smooth YMH field on $B^*_{R_0}$, where $R_0\leq 1$, satisfying the energy bound \eqref{inq3'}, then there exists $r_0\leq R_0/2$, depending only on $\ep_0$ and $n$, such that $(F_A, \n_A u)$ satisfies
\begin{align}
\sqrt{|F_A|^2(x)+1}\leq C(1+\ep_0^{2}r_0^{-2}),\label{key-es-1}\\
\sqrt{|\n_A u|^2(x)+1}\leq C(1+\ep_0r^{-1}_0),\label{key-es-2}
\end{align} 
for any $0<r=|x|<r_0$.
\end{thm}

%
%

\begin{rem}
Theorem \ref{remov-singu} generalizes the classical removable singularity theorems for coupled Yang-Mills fields \cite{S85,OSib} and harmonic maps \cite{L85}, and extends the three-dimensional result for YMH fields \cite{CH1}.
\end{rem}

\subsection{Outline of the proof} We outline the proof of Theorem \ref{main-thm}; the remaining results follow directly from it. The proof must overcome challenges arising from both the nonflatness of the fiber and the lack of conformal invariance of the Yang-Mills-Higgs field equations. When the fiber is a general compact Riemannian manifold, the YMH equations are highly nonlinear compared with those for vector-valued Higgs fields. This creates new difficulties in analyzing the asymptotic behavior of YMH fields near singularities. Our main tool for overcoming this obstacle is the extended Kato inequalities for $n$-dimensional YMH fields $(A,u)$ established in \cite{CH1}. These inequalities are
\begin{align}
	|\n_A \n_A u|^2\geq &(\frac{n}{n-1}-\de)|\n |\n_Au||^2-C(n, \de)(1+|F_A|^2),\label{Kato1}\\
	|\n_AF_A|^2\geq &(\frac{n-1}{n-2}-\de)|\n |F_A||^2-C(n, \de)|u^*(\n_A u)|^2,\label{Kato2}
\end{align}
for any sufficiently small $\de>0$. 

More precisely, motivated by \cite{Ra93,SU19}, we combine the Bochner formulas \eqref{Bochner-1} and \eqref{Bochner-2} with the Kato inequalities \eqref{Kato1}-\eqref{Kato2} to derive differential inequalities for YMH fields $(A,u)$. We find that
 \[(f,g)=\((|F_A|^2+1)^{\frac{\beta}{2}},(|\n_Au|^2+1)^{\frac{\ka}{2}}\)\]
satisfy the following key differential inequalities
\begin{align}
	\De f\geq& -C(n,\de)(|\n_A u|^2+|F_A|+1)f,\label{finer-eq-YMH-1}\\
	\De g\geq&-C(n,\de)(|\n_A u|^2+|F_A|^2+1)g,\label{finer-eq-YMH-2}	
\end{align}
for every sufficiently small $\de>0$. Here we set $\beta=\frac{n-3}{n-2}+\de$ and $\ka=\frac{n-2}{n-1}+\de$. By studying decay estimates for positive solutions of \eqref{finer-eq-YMH-1} and \eqref{finer-eq-YMH-2}, we establish estimates \eqref{key-es-1}-\eqref{key-es-2} for the YMH field $(A,u)$ in Theorem \ref{main-thm}. We rewrite these differential inequalities in cylindrical coordinates, reducing the study of decay near a point singularity to the analysis of the asymptotic behavior at infinity of solutions to the corresponding inequalities.

In the special case in which $A$ is a four-dimensional YM field, the Kato inequality for the curvature $F_A$,
\[\frac{3}{2}|\n|F_A||^2\leq |\n_AF_A|^2,\]
implies that $F_A$ satisfies a key inequality with respect to the Euclidean metric $g_{euc}$:
\begin{align}
\De_{g_{euc}} |F_A|^{\frac{1}{2}}_{g_{euc}}-(K_{euc}-C|F_A|_{g_{euc}})|F_A|^{\frac{1}{2}}_{g_{euc}}\geq 0, \label{key-ym0} 
\end{align}
where $K_{euc}$ denotes a constant depending only on the curvature of Euclidean space. Because Euclidean space is flat, $K_{euc}=0$. In cylindrical coordinates, this inequality becomes
\begin{equation}\label{key-ym}
(\p_t^2+\De_{\mathbb{S}^3})|F_A|^{\frac{1}{2}}_{g_{C}}-(1-C|F_A|_{g_C})|F_A|^{\frac{1}{2}}_{g_{C}}\geq 0,
\end{equation}
where $g_C=e^{2t}g_{\mathrm{euc}}$ is the canonical cylindrical metric, $t=-\log r$, and $r=|x|$. By studying inequality \eqref{key-ym}, R\aa{}de \cite{Ra93} established a sharp decay estimate for $F_A$.

Formulas \eqref{key-ym0} and \eqref{key-ym} have the same structure because Yang-Mills fields in four dimensions are conformally invariant. In general, however, the YMH equations are not conformally invariant, so inequalities \eqref{finer-eq-YMH-1} and \eqref{finer-eq-YMH-2} do not retain their form in cylindrical coordinates. To overcome this difficulty, we introduce suitable weighted modifications of $f$ and $g$ so that the resulting inequalities take refined forms analogous to \eqref{key-ym} in cylindrical coordinates (see \eqref{ineq-f-bar} and \eqref{ineq-g-bar} below).
 
The proof proceeds as follows. We first establish almost $L^\infty$ estimates for the curvature. Because the YMH field $(A,u)$ satisfies \eqref{inq3'}, the $\ep$-regularity theorem, Theorem \ref{ep-reg}, implies that
\begin{equation}\label{eq-ep-reg}
	r^2(|\n_A u|^2+|F_A|)\le C\ep^2_0.
\end{equation}

Next, using \eqref{finer-eq-YMH-1}, we observe that the weighted function $\bar{f}=e^{-\frac{n-2}{2}t}f$ in cylindrical coordinates $(t=-\log|x|, \th)\in [t_0=-\log r_0,+\infty)\times \mathbb{S}^{n-1}$ satisfies
\begin{equation}\label{ineq-f-bar}
	\p^2_t\bar{f}+\De_{\mathbb{S}^{n-1}}\bar{f}-\al^2\bar{f}\geq 0,
\end{equation}
where $\al$ is defined as 
\[\al=\sqrt{\frac{(n-2)^2}{4}-C(n)(r^{2/3}_0+\ep^2_0)}.\]
If $n\geq 4$ and $\de$, $r_0$, and $\ep_0$ are sufficiently small, then $2\beta-\frac{n-2}{2}<\al$. Applying comparison theorems to \eqref{ineq-f-bar}, we obtain almost $L^\infty$ estimates for $F_A$ (see \eqref{Decay-es-ymh-1}).

With this almost $L^\infty$ estimate for $F_A$ in hand, we analyze the weighted function $\bar{g}=e^{-\frac{n-2}{2}t}g$. It follows from \eqref{finer-eq-YMH-2} that $\bar{g}$ satisfies
\begin{equation}\label{ineq-g-bar}
	\p^2_t \bar{g}+\De_{\mathbb{S}^{n-1}}\bar{g}-\al^2\bar{g}\geq -Ce^{-\varrho(t-t_0)}\bar{g},
\end{equation}
where $\varrho$ is a constant close to $2$. Applying the comparison theorems to \eqref{ineq-g-bar}, we derive an almost $L^\infty$ estimate for $|\n_A u|$, namely \eqref{Decay-es-ymh-2}.

Once these almost $L^\infty$ estimates are established, we refine them to the sharp $L^\infty$ estimates \eqref{key-es-1} and \eqref{key-es-2} by applying a similar analysis to the differential inequalities \eqref{finer-eq-YMH-1} and \eqref{finer-eq-YMH-2}.

\medskip
The remainder of the paper is organized as follows. In Section \ref{s: pre}, we recall preliminary results for YMH fields. In Section \ref{s: decay-es}, we prove Theorem \ref{main-thm}. Finally, in Section \ref{s: remov-singu}, we prove Theorem \ref{remov-singu}.

\section{Preliminaries}\label{s: pre}

\subsection{Bundle notation and equations for YMH fields}
Let $(N,h)$ be a compact Riemannian manifold equipped with an effective action of a compact connected Lie group $G$ that preserves the metric $h$. For fixed $R>0$, let $B_R\subset \Real^n$ be the Euclidean ball of radius $R$, and let $B^*_R=B_R\setminus\{0\}$ be the punctured ball. Let $\P$ be a trivial principal $G$-bundle over $B_R$, and let $\F=\P\times_G N$ be the associated bundle. After fixing a trivialization, any section of $\F$ can be identified with a map $u:B_R\to N$, while a connection $A$ on $\P$ can be identified with a $\g$-valued 1-form, where $\g$ is the Lie algebra of $G$.

A connection $A$ induces an extrinsic derivative $D_A=d+A$, and the curvature of $A$ is defined by
\[F_A=D^2_A=dA+\frac{1}{2}[A,A].\]
The connection $A$ also induces a covariant derivative $\n_A=\n+A$ such that for any section $u: B_R\to N$ we have
\[\n_A u=\n u+A\cdot u,\]
where $\cdot$ denotes the infinitesimal action of $\g$ on $N$.

Let $s: B_R \to G$ be a gauge transformation. The action of $s$ on a connection $A$, a section $u$, and the curvature $F_A$ of $A$ is given by the following transformation laws:
$$s^{*}A=s^{-1}ds+s^{-1}As,\, s^{*}F_{A}=F_{s^{*}A}=s^{-1}F_{A}s,$$
and
$$s^*D_A u=D_{s^*A}s^*u=s^{-1}D_{A}u,\, V(s^*u)=V(u).$$
These transformation laws imply that the YMH functional \eqref{YMHf} is invariant under the gauge transformation $s$; that is,
$$\mathscr{E}(A,u)=\mathscr{E}(s^{*}A,s^{*}u).$$
The YMH equations are therefore gauge invariant. Namely, if $(A,u)$ satisfies the system
\begin{equation}\label{eq-ymh1}
	\begin{cases}
		D^*_A F_A=-u^* \n_Au, \\[1ex]
		\n_A^* \n_Au=-2\n V(u),
	\end{cases}
\end{equation}
then $(s^*A,s^* u)$ also solves equation \eqref{eq-ymh1} for any gauge transformation $s$.

For the purposes of PDE analysis, we rewrite equation \eqref{eq-ymh1} in a more explicit extrinsic form. By the equivariant embedding theorem of Moore and Schlafly~\cite{MS}, there exist an isometric embedding
$i:N\rightarrow\mathbb{R}^{K}$ and a representation $\rho:G\longrightarrow SO(K)$ such that $i(g\cdot u)=\rho(g)i(u)$ for every $u\in N$ and $g\in G$. Under this representation, the Lie algebra $\mathfrak{g}$ corresponds to a subalgebra of $\mathfrak{so}(K)$, the space of skew-symmetric $K\times K$ matrices. Thus, for any $c\in \mathfrak{g}$ and $u\in N \hookrightarrow \mathbb{R}^{K}$, the infinitesimal action of $c$ on $u$ is
$$
c\cdot u=X_{c}(u)=\rho(c)u,
$$
where $X_c$ is the infinitesimal vector field induced by $c$, and $\rho(c)u$ denotes the usual action of the matrix $\rho(c)$ on the vector $u$. It follows that the action of $c$ on a vector field $Y\in \Ga(TN)$ is
$$
c\cdot Y=\nabla_{Y}X_{c}=(\rho(c)\cdot Y)^{\top}=\rho(c) Y-\Ga(y)(X_{c}, Y),
$$
where $\top$ denotes the projection from $\mathbb{R}^K$ onto the tangent space of $N$, and $\Ga$ denotes the second fundamental form of $N$ in $\Real^K$.

Using this notation, we can rewrite equation \eqref{eq-ymh1} as
\begin{equation}\label{eq-ymh2}
	\begin{cases}
		d^{*}dA+A\#dA +A\#A\#A=-u^{*}(\n_{A}u),\\
		\Delta u-d^{*}Au+2A\#du+A\#(Au)=\Ga(u)(\n_Au,\n_Au)+2\n V(u).
	\end{cases}
\end{equation}
Here, for simplicity, we continue to use $A$ to denote $\rho(A)$, we set $\n_Au=du+Au$, and $\#$ denotes a linear contraction.

\subsection{Bochner formulas and $\ep$-regularity theorems for YMH fields}
Let $(A,u)$ be a YMH field on $B_R\subset \Real^n$ with $n\geq 4$. We recall the Bochner formulas and the $\ep$-regularity theorem obtained in our previous work \cite{CH1}.

\begin{lem}\label{bochner}
There exist constants $a$, $b$, and $c$, depending only on $n$ and the geometry of $N$, such that every YMH field $(A,u)$ defined on $B_R$ has the following properties:
\begin{itemize}
\item [$(1)$] $(\n_A u, F_A)$ satisfies
\begin{align}
\De |\n_A u|^2&\geq |\n_A\n_A u|^2-a(|\n_A u|^2+|F_A|+1)|\n_A u|^2,	\label{Bochner-1}\\
\De |F_A|^2&\geq |\n_A F_A|^2-b(|\n_A u|^2+|F_A|^2)|F_A|;\label{Bochner-2}
\end{align}
\item[$(2)$] If $f=|F_A|+|\n_A u|^2$, then
\[\De f\geq -c(1+f)f.\]
\end{itemize}
\end{lem}

\begin{thm}\label{ep-reg}
There exists $\ep_0>0$ such that for any YMH field $(A,u)$ on $B_R$ satisfying
\begin{equation}\label{Mor-c}
\sup_{B_{\rho(y)}\subset B_{R}}\(\frac{1}{\rho^{n-2}}\int_{B_{\rho}(y)}|\n_A u|^2+|F_{A}|dx\)\leq \ep^2_0,
\end{equation}	
we have
\[\sup_{B_{R/2}}R^2(|\n_A u|^2+|F_A|)\leq C\ep^2_0\]
for some constant $C$ depending only on $n$ and $\sup_{N}|R^N|$.
\end{thm}

A direct corollary of Theorem \ref{ep-reg} is as follows.
\begin{cor}
There exists $\ep_0>0$ such that for any YMH field $(A,u)$ on $B_R$ satisfying
\begin{equation}\label{inq1}
\int_{B_{R}}|\n_A u|^n+|F_{A}|^{n/2}dx\leq \ep^n_0,
\end{equation}	
we have
 \[\sup_{B_{R/2}}R^2(|\n_A u|^2+|F_A|)\leq C\ep^2_0\]
for some constant $C$ depending only on $n$ and $\sup_{N}|R^N|$.
\end{cor}
\subsection{Improved Kato inequalities for YMH fields}
We recall the Kato inequalities for YMH fields established in \cite{CH1}, which play a crucial role in obtaining decay estimates in the next section. Let $(A,u)$ be a smooth YMH field on $B_R$; that is, $(A,u)$ satisfies
\begin{equation*}
	\begin{cases}
		D^*_AF_A=-u^*(\n_Au),\\[1ex]
		\n^*_A\n_A u=-2\n V(u),
	\end{cases}
\end{equation*}
where $V(u)$ is a smooth function of the section $u$. By the definition of curvature, $u$ satisfies
\[D^*_A D_A u=-2\n V(u), \quad D_AD_A u=F_A\cdot u.\]
The Bianchi identity also gives
\[D^*_A F_A=-u^*(\n_A u),\quad D_AF_A=0.\]

Thus, Propositions 3.2--3.4 of \cite{CH1} yield almost sharp Kato inequalities for the YMH field $(A,u)$, as stated in the following proposition.

\begin{prop}\label{Kato-ineq-4}
For every $\de>0$, there exists a constant $C(n,\de)>0$ such that every YMH field $(A,u)$ on $B_R$ satisfies the following inequalities:
\begin{itemize}
\item[$(1)$] the section $u$ satisfies
\begin{align}
|\n_A \n_A u|^2\geq &(\frac{n}{n-1}-\de)|\n |\n_Au||^2-C(n,\de)(1+|F_A|^2);\label{Kato-1}
\end{align}
\item[$(2)$] the curvature $F_A$ of $A$ satisfies
\begin{align}
|\n_AF_A|^2\geq (\frac{n-1}{n-2}-\de)|\n |F_A||^2-C(n, \de)|u^*(\n_A u)|^2.\label{Kato-2}
\end{align}
\end{itemize}
Moreover, if $A$ is a pure Yang-Mills field, then 
\begin{align}
 |\n_A F_A|^2\geq \frac{n-1}{n-2}|\n |F_A||^2; \label{Kato-YM}  
\end{align}
if $u$ is a harmonic map, then
\begin{align}
 | \n^2 u|^2\geq \frac{n}{n-1}|\n |\n u||^2.  \label{Kato-HM} 
\end{align}
\end{prop}

\subsection{Differential inequalities for YMH fields} Combining the Bochner formulas in Lemma \ref{bochner} with the improved Kato inequalities in Proposition \ref{Kato-ineq-4}, we obtain differential inequalities for $n$-dimensional Yang-Mills-Higgs fields. These inequalities are essential to the proof of Theorem \ref{main-thm}.
\begin{prop}\label{eq-YMH}
For every $\de>0$, there exists a constant $C(n,\de)>0$ such that every YMH field $(A,u)$ on $B^*_{R_0}$ has the following properties.

\begin{itemize}
\item[$(1)$] For $p\geq \frac{n-3}{n-2}+\de$, the function $f=\sqrt{|F_A|^2+1}$ satisfies 
\begin{align}
\De f^p\geq -C(n, \de)(|\n_A u|^2+f)f^p.\label{finer-eq-YMH1}
\end{align}
\item[$(2)$] For $p\geq \frac{n-2}{n-1}+\de$, the function $f=\sqrt{|\n_A u|^2+1}$ satisfies 
\begin{align}
\De f^p\geq -C(n, \de)(|F_A|^2+f^2)f^p.\label{finer-eq-YMH2}	
\end{align}
\end{itemize}
\end{prop}

\begin{proof}
For any smooth positive function $f$, a direct calculation gives
\begin{equation}\label{eq-f}
\De f^p=pf^{p-2}\(\frac{1}{2}\De f^2+\frac{p-2}{4}\frac{|\n f^2|^2}{f^2}\).
\end{equation}

Taking $f=\sqrt{|F_A|^2+1}$ in \eqref{eq-f}, with $p>0$, gives
\begin{align*}
	\De f^p=&pf^{p-2}\(\frac{1}{2}\De |F_A|^2+\frac{p-2}{4}\frac{|\n |F_A|^2|^2}{f^2}\)\\
	=& pf^{p-2}(\frac{1}{2}\De |F_A|^2+(p-2)|\n |F_A||^2\frac{|F_A|^2}{f^2}).
\end{align*}

Applying the Bochner formula \eqref{Bochner-2} from Lemma \ref{bochner} and the Kato inequality \eqref{Kato-2} for $F_A$ from Proposition \ref{Kato-ineq-4}, we obtain
\begin{align*}
	\frac{1}{2}\De|F_A|^2\geq& |\n_A F_A|^2-C|F_A||\n_Au|^2-C(1+|F_A|)|F_A|^2,\\
	\geq & (\frac{n-1}{n-2}-\de)|\n |F_A||^2-C(n, \de)(1+|F_A|)(|\n_Au|^2+|F_A|^2).
\end{align*}
This immediately yields
\begin{align*}
	\De f^p=& pf^{p-2}(\frac{1}{2}\De |F_A|^2+(p-2)|\n |F_A||^2\frac{|F_A|^2}{f^2})\\
	\geq &pf^{p-2}\(\frac{n-1}{n-2}-\de+(p-2)\)|\n|F_A||^2\frac{|F_A|^2}{f^2}\\
	&-C(n,\de)f^{p-1}(|\n_A u|^2+f^2),
\end{align*}
where we used the assumption $p\geq 2+\de-\frac{n-1}{n-2}$. Since $f^{-1}\leq 1$, it follows that
\[\De f^p\geq -C(n, \de)(|\n_A u|^2+f)f^p.\]

Similarly, taking $f=\sqrt{|\n_A u|^2+1}$ in \eqref{eq-f}, with $p>0$, gives
\[\De f^p=pf^{p-2}(\frac{1}{2}\De |\n_A u|^2+(p-2)|\n |\n_A u||^2\frac{|\n_A u|^2}{f^2}).\]
Applying the Bochner formula \eqref{Bochner-1} from Lemma \ref{bochner} and the Kato inequality \eqref{Kato-1} for $\n_A u$ from Proposition \ref{Kato-ineq-4}, we obtain
\begin{align*}
\frac{1}{2}\De|\n_A u|^2
\geq &|\n_A \n_A u|^2+|u^*(\n_Au)|^2-C(1+|F_A|+|\n_A u|^2)|\n_A u|^2\\
\geq &(\frac{n}{n-1}-\de)|\n |\n_Au||^2-C(|F_A|^2+f^2)f^2.
\end{align*}
Consequently, for $p\geq 2+\de-\frac{n}{n-1}$, we have
\begin{align*}
\De f^p
   \geq -C(n, \de)(|F_A|^2+f^2)f^p.
\end{align*}

This completes the proof.
\end{proof}

For later use, we rewrite differential inequalities \eqref{finer-eq-YMH1} and \eqref{finer-eq-YMH2} in cylindrical coordinates. The conformal transformation from $B_{R_0}^*$ to the cylinder $\mathscr{C}=[-\log R_0,+\infty)\times \mathbb{S}^{n-1}$ is given by
\[\phi: B_{R_0}^*\to \mathscr{C},\quad (r,\th)\to (t=-\log r, \th),\]
where the cylinder $\C$ is equipped with the canonical metric $g=dt^2+g_{\mathbb{S}^{n-1}}$. For simplicity, we denote the Laplace operator on $\mathbb{S}^{n-1}$ by $\De_{\mathbb{S}^{n-1}}$. 
\begin{prop}\label{finer-eq-cylinder}
For every $\de>0$, there exists a constant $C(n,\de)>0$ such that, for every YMH field $(A,u)$ on $B^*_{R_0}$, the functions
\[(g,\eta)=((1+|F_A|^2)^{\frac{\beta}{2}},(1+|\n_A u|^2)^{\frac{\ka}{2}})\]
satisfy
\begin{align}
\p^2_t g-(n-2)\p_t g+\De_{\mathbb{S}^{n-1}}g\geq &-C(n,\de)e^{-2t}(|\n_A u|^2+|F_A|+1)g,\label{cylinder-eq-1}\\
\p^2_t \eta-(n-2)\p_t \eta+\De_{\mathbb{S}^{n-1}}\eta\geq &-C(n,\de)e^{-2t}(|\n_A u|^2+|F_A|^2+1)\eta,\label{cylinder-eq-2}
\end{align}
where $\beta=\frac{n-3}{n-2}+\de$ and $\ka=\frac{n-2}{n-1}+\de$.
\end{prop}
\begin{proof}
Inequalities \eqref{cylinder-eq-1} and \eqref{cylinder-eq-2} are precisely the cylindrical-coordinate forms of \eqref{finer-eq-YMH1}, with $p=\beta$, and \eqref{finer-eq-YMH2}, with $p=\ka$, respectively.
\end{proof}

\section{Decay estimates for YMH fields in higher dimensions}\label{s: decay-es}
In this section, we prove Theorem \ref{main-thm}. We begin with a comparison theorem used to obtain decay estimates for $n$-dimensional YMH fields under the energy bound \eqref{inq3'}.
\begin{lem}\label{Key}
Let $\bar{g}$ be a smooth solution to the following differential inequality 
\[\p^2_t \bar{g}+\De_{\mathbb{S}^{n-1}}\bar{g}-\al^2\bar{g}\geq 0\]
on $[T_0,\infty)\times \mathbb{S}^{n-1}$, where $\al>0$. Suppose that $\bar{g}$ satisfies the upper bound
\[\bar{g}\leq c_1e^{-\ld_1t}+c_2e^{\ld_2t}\]
for any $(t,\th)\in [T_0, \infty)\times \mathbb{S}^{n-1}$, where $\ld_1\geq 0$ and $0\leq \ld_2<\al$. Then for any $(t,\th)\in [T_0, \infty)\times \mathbb{S}^{n-1}$, we have
\[\bar{g}(t,\th)\leq (c_1e^{-\ld_1T_0}+c_2e^{\ld_2T_0})e^{-\al(t-T_0)}.\]
\end{lem}
\begin{proof}
For $T>T_0$, consider the comparison function
\[h(t)=(c_1e^{-\ld_1T_0}+c_2e^{\ld_2T_0})e^{-\al(t-T_0)}+(c_1e^{-\ld_1T}+c_2e^{\ld_2T})e^{-\al(T-t)},\]
which satisfies the equation
\[\p^2_t h+\De_{\mathbb{S}^{n-1}}h-\al^2h=0\]
and the boundary conditions
\[\bar{g}(T_0)\leq h(T_0),\qquad \bar{g}(T)\leq h(T).\]
	
Therefore, the comparison principle gives
\[\bar{g}(t,\th)\leq h(t)\]
on $[T_0,T]\times \mathbb{S}^{n-1}$. Letting $T\to +\infty$ and using $\al>\ld_2$, we obtain
\[\bar{g}(t,\th)\leq (c_1e^{-\ld_1T_0}+c_2e^{\ld_2T_0})e^{-\al(t-T_0)}\]
for any $(t,\th)\in [T_0, \infty)\times \mathbb{S}^{n-1}$.

This completes the proof.
\end{proof}

Next, we establish almost $L^\infty$ estimates for YMH fields satisfying the energy bound \eqref{inq3}.
\begin{thm}\label{main-thm-represent}
There exists $\ep_0>0$ such that if $(A,u)$ is a smooth YMH field defined on $B^*_{R_0}$, where $R_0\leq 1$, satisfying
\begin{equation}\label{inq3}
	\sup_{B_{\rho(y)}\subset B_{R_0}}\(\frac{1}{\rho^{n-2}}\int_{B_{\rho}(y)}|\n_A u|^2+|F_A|dx\)\leq \ep^2_0,
\end{equation}	
then there exists an $r_0\leq R_0/2$ such that $(A,u)$ satisfies the following decay estimates
\begin{align}
(1+|F_A|^2)^{\frac{\beta}{2}}(r,\th)\leq& C(1+\ep_0^{2\beta}r_0^{-2\beta})\(\frac{r}{r_0}\)^{(\al-\frac{n-2}{2})},\label{Decay-es-ymh-1}\\
\(1+|\n_A u|^2\)^{\ka/2}(r,\th)\leq &C(1+\ep^\ka_0r^{-\ka}_0)\(\frac{r}{r_0}\)^{(\al-\frac{n-2}{2})},\label{Decay-es-ymh-2}
\end{align}
for any $0<r<r_0$, where we denote
\[\al=\sqrt{\frac{(n-2)^2}{4}-C(\ep^2_0+r^{2/3}_0)},\]
$\beta=1-\frac{1}{2(n-2)}$ and $\ka=1-\frac{1}{2(n-1)}$.
\end{thm}
\begin{proof}
The proof is divided into two steps.

\medskip
\noindent\emph{Step 1: Decay estimate for $F_A$.}

By the energy bound \eqref{inq3}, for any $x\in B^*_{2R_0/3}$ and $r=|x|$ we have
\[\sup_{B_{\rho}(y)\subset B_{r/2}(x)}\(\frac{1}{\rho^{n-2}}\int_{B_{\rho}(y)}|\n_A u|^2+|F_A|dx\)\leq \ep^2_0.\]
The $\ep$-regularity theorem, Theorem \ref{ep-reg}, then yields
\begin{equation}\label{thm-es-ymh}
r^2(|\n_A u|^2+|F_A|)(r,\th)\leq C\ep^2_0    
\end{equation}
for any $0<r=|x|\leq R_0/2$. 

We now work in cylindrical coordinates, where $t=-\log r$ and
\[(t,\th)\in [-\log (R_0/2),+\infty)\times \mathbb{S}^{n-1}.\]
In these coordinates, estimate \eqref{thm-es-ymh} becomes
\begin{equation}\label{thm-es-ymh-1}
 e^{-2t}(|\n_A u|^2+|F_A|)(t,\th)\leq C\ep^2_0.
\end{equation}
It follows from inequality \eqref{cylinder-eq-1} in Proposition \ref{finer-eq-cylinder} that the function $g=(1+|F_A|^2)^{\frac{\beta}{2}}$, where $\beta=\frac{n-3}{n-2}+\de$, satisfies
\begin{align*}
\p^2_t g-(n-2)\p_t g+\De_{\mathbb{S}^{n-1}}g\geq &-C(n,\de)(e^{-2t}+\ep^2_0)g
\end{align*}
on $[-\log (R_0/2),+\infty)\times \mathbb{S}^{n-1}$. Hence, the weighted function $\bar{g}(t,\th)=e^{-\frac{n-2}{2}t}g$ satisfies
\begin{equation}\label{weight-finer-eq-ymh-1}
 \p^2_t \bar{g}+\De_{\mathbb{S}^{n-1}}\bar{g}-(\frac{(n-2)^2}{4}-C(\ep^2_0+e^{-2t}))\bar{g}\geq 0.   
\end{equation}

Fix $\de=\frac{1}{2(n-2)}$ and choose $t_0>-\log R_0$ such that
\begin{align*}
\al=&\sqrt{\frac{(n-2)^2}{4}-C(e^{-2t_0}+\ep^2_0)}\\
>&2\beta-\frac{n-2}{2}=2-\frac{1}{n-2}-\frac{n-2}{2}  
\end{align*}
for $n\geq 4$. Then, for $t>t_0$, we have
\begin{equation}\label{weight-finer-eq-ymh-2}
\p^2_t \bar{g}+\De_{\mathbb{S}^{n-1}}\bar{g}-\al^2\bar{g}\geq 0, 
\end{equation}
Moreover, estimate \eqref{thm-es-ymh-1} gives
\[\bar{g}(t,\th)\leq C(e^{-\frac{n-2}{2}t}+\ep_0^{2\beta}e^{(2\beta-\frac{n-2}{2})t})\]
for any $(t,\th)\in [t_0, \infty)\times \mathbb{S}^{n-1}$. 

For any $m\in \mathbb{N}^+$, applying Lemma \ref{Key} with $T_0=mt_0$ to differential inequality \eqref{weight-finer-eq-ymh-2} gives
\[\bar{g}(t,\th)\leq C(e^{-\frac{n-2}{2}mt_0}+\ep_0^{2\beta}e^{(2\beta-\frac{n-2}{2})mt_0})e^{-\al(t-mt_0)}.\]
Equivalently,
\begin{align}
  (1+|F_A|^2)^{\frac{\beta}{2}}(t,\th)\leq C(1+\ep_0^{2\beta}e^{2\beta mt_0})e^{-(\al-\frac{n-2}{2})(t-mt_0)}.  \label{thm-es-FA}
\end{align}
for any $(t,\th)\in [mt_0, \infty)\times \mathbb{S}^{n-1}$.

Taking $r_0=e^{-3t_0}$ and expressing estimate \eqref{thm-es-FA} in polar coordinates yields the desired decay estimate for $F_A$.

\medskip
\noindent\emph{Step 2: Decay estimate for $\n_A u$.}

Taking $\de=\frac{1}{2(n-1)}$ in inequality \eqref{cylinder-eq-2}, we find that the function $\eta=(1+|\n_A u|^2)^{\frac{\ka}{2}}$ satisfies
\begin{equation}\label{eq-eta}
\begin{aligned}
&\p^2_t \eta-(n-2)\p_t \eta+\De_{\mathbb{S}^{n-1}}\eta\\
\geq& -Ce^{-2t}(|\n_A u|^2+|F_A|^2+1)\eta,
\end{aligned}
\end{equation}
where $\ka=\frac{n-2}{(n-1)}+\de=1-\frac{1}{2(n-1)}$.

By estimates \eqref{thm-es-ymh-1} and \eqref{thm-es-FA}, on $[t_0,\infty)\times \mathbb{S}^{n-1}$ we have
\begin{align}
 e^{-2t}|\n_A u|^2\leq& C\ep_0^2,\label{decay-u-new}\\
 e^{-2t}|F_A|^2\leq & C\Lambda(\ep_0)e^{-\varrho(t-t_0)},\label{improved-decay-F}
\end{align}
where $\Lambda(\ep_0)=(e^{-2t_0}+2\ep^2_0+\ep^4_0e^{2t_0})$ and $\varrho=2+\frac{2}{\beta}(\al-\frac{n-2}{2})\geq 1+\ka$.

Using estimates \eqref{decay-u-new} and \eqref{improved-decay-F} in inequality \eqref{eq-eta}, we find that the weighted function $\bar{\eta}=e^{-\frac{n-2}{2}t}\eta$ satisfies
\begin{equation}\label{eq-bar-eta}
\p^2_t \bar{\eta}+\De_{\mathbb{S}^{n-1}}\bar{\eta}-\al^2\bar{\eta}\geq -C\Lambda(\ep_0)e^{-\varrho(t-t_0)}\bar{\eta}
\end{equation}
on $[t_0, \infty)\times \mathbb{S}^{n-1}$. Moreover, estimate \eqref{decay-u-new} implies that
\begin{align}\label{thm-es-u}
\bar{\eta}\leq C(e^{-\frac{n-2}{2}t}+\ep_0^{\ka}e^{(\ka-\frac{n-2}{2})t})    
\end{align}
for every $(t,\th)\in [t_0,\infty)\times \mathbb{S}^{n-1}$. Thus, the right-hand side of inequality \eqref{eq-bar-eta} can be estimated as follows:
\[Q=-C\Lambda(\ep_0) e^{-\varrho(t-t_0)}\bar{\eta}\geq -\tilde{\eta} \]
where
\begin{align*}
 \tilde{\eta}=&C\Lambda(\ep_0) e^{-\varrho(t-t_0)}(e^{-\frac{n-2}{2}t}+\ep_0^{\ka}e^{(\ka-\frac{n-2}{2})t})\\
 \leq & C(e^{-\frac{n-2}{2}t}+\ep_0^{\ka}e^{(\ka-\frac{n-2}{2})t})
\end{align*}
for $t\geq (\frac{2}{\varrho}+1)t_0$.

For $T>3t_0$, consider the function
\begin{align*}
 h(t)=&2C(e^{-\frac{n-2}{2}3t_0}+\ep_0^{\ka}e^{(\ka-\frac{n-2}{2})3t_0})e^{-\al(t-3t_0)}\\
 &+2C(e^{-\frac{n-2}{2}T}+\ep_0^{\ka}e^{(\ka-\frac{n-2}{2})T})e^{-\al(T-t)}-\tilde{\eta}   
\end{align*}
as a comparison function for the differential inequality
\[\p^2_t \bar{\eta}+\De_{\mathbb{S}^{n-1}}\bar{\eta}-\al^2\bar{\eta}\geq -\tilde{\eta}.\]

A direct calculation shows that
\begin{align*}
(\p^2_t +\De_{\mathbb{S}^{n-1}}-\al^2)h\leq -\((\varrho+\frac{n-2}{2}-\ka)^2-\al^2\)\tilde{\eta}\leq -\tilde{\eta},
\end{align*}
which implies 
\[(\p^2_t +\De_{\mathbb{S}^{n-1}}-\al^2)(h-\bar{\eta})\leq 0\]
on $[3t_0,T]\times \mathbb{S}^{n-1}$. By estimate \eqref{thm-es-u}, the function $h$ also satisfies the boundary conditions
$h(3t_0)\geq \bar{\eta}(3t_0,\th)$ and $h(T)\geq \bar{\eta}(T,\th)$. Consequently, the comparison principle gives
\[\bar{\eta}(t,\th)\leq h(t)\]
for $(t,\th)\in [3t_0,T]\times\mathbb{S}^{n-1}$. Letting $T\to +\infty$ yields the desired estimate for $\n_Au$.

This completes the proof.
\end{proof}

Using the almost $L^\infty$ estimates established in Theorem \ref{main-thm-represent}, we now prove Theorem \ref{main-thm}, which we restate in the following form.

\begin{thm}\label{sharp-es}
There exists $\ep_0>0$ such that if $(A,u)$ is a smooth YMH field defined on $B^*_{R_0}$, where $R_0\leq 1$, satisfying the energy bound \eqref{inq3}, then there exists $r_0\leq R_0/2$ such that $(A,u)$ satisfies the following $L^\infty$ estimates:
\begin{align}
	(1+|F_A|^2)^{\frac{1}{2}}(r,\th)\leq& C(1+\ep_0^{2}r_0^{-2})\label{Decay-es-ymh-1'}\\
	\(1+|\n_A u|^2\)^{\frac{1}{2}}(r,\th)\leq &C(1+\ep_0r^{-1}_0)\label{Decay-es-ymh-2'}
\end{align}
for all $0<r<r_0$.
\end{thm}
\begin{proof}
By estimates \eqref{Decay-es-ymh-1} and \eqref{Decay-es-ymh-2}, together with inequality \eqref{cylinder-eq-1}, we infer that the function $g=(1+|F_A|^2)^{\frac{\beta}{2}}$, where $\beta=\frac{n-3}{n-2}+\frac{1}{2(n-2)}$, satisfies
\begin{align*}
\p^2_t g-(n-2)\p_t g+\De_{\mathbb{S}^{n-1}}g\geq &-CL(\ep_0, t_0)e^{-(2+3(\al-\frac{n-2}{2}))t},
\end{align*}
on $\mathbb{S}^{n-1}\times [3t_0,\infty)$, where $L(\ep_0,t_0)=1+\ep^2_0e^{-3t_0}$.

Consequently, the function $\bar{g}(t,\th)=e^{-\frac{n-2}{2}t}g$ satisfies
\begin{equation*}
	\p^2_t \bar{g}+\De_{\mathbb{S}^{n-1}}\bar{g}-\frac{(n-2)^2}{4}\bar{g}\geq -CL(\ep_0, t_0)e^{-(\frac{(n-2)}{2}+2+3(\al-\frac{n-2}{2}))t}.   
\end{equation*}

Next, we consider the comparison function
\begin{align*}
	g_1(t)=&2C(e^{-\frac{n-2}{2}3t_0}+\ep_0^{2\beta}e^{(2\beta-\frac{n-2}{2})3t_0})e^{-\frac{n-2}{2}(t-3t_0)}\\
	&+2C(e^{-\frac{n-2}{2}3t_0}+\ep_0^{2\beta}e^{(2\beta-\frac{n-2}{2})3t_0})e^{-\al(T-3t_0)}e^{-\frac{n-2}{2}(T-t)}\\
	&-CL(\ep_0, t_0)e^{-(\frac{(n-2)}{2}+2+3(\al-\frac{n-2}{2}))t}.  
\end{align*}
For a suitable choice of $t_0$, a direct computation yields
\begin{align*}
\(\p^2_t+\De_{\mathbb{S}^{n-1}}-\frac{(n-2)^2}{4}\)(\bar{g}-g_1)\geq 0
\end{align*}
on $\mathbb{S}^{n-1}\times [3t_0,T]$, together with the boundary conditions
\[\bar{g}(3t_0)\leq g_1,\quad \bar{g}(T)\leq g_1(T).\]
Therefore, by the comparison principle, we obtain
\[\bar{g}(t,\th)\leq g_1(t)\]
for $3t_0\leq t\leq T$.

Since $\al<\frac{n-2}{2}$, letting $T\to +\infty$ yields
\[\bar{g}(t,\th)\leq 2C(e^{-\frac{n-2}{2}3t_0}+\ep_0^{2\beta}e^{(2\beta-\frac{n-2}{2})3t_0})e^{-\frac{n-2}{2}(t-3t_0)}\]
for every $t\geq 3t_0$. Transforming back to polar coordinates gives the desired estimate \eqref{Decay-es-ymh-1'}.

Finally, applying the same argument to the function $\eta=(1+|\n_A u|^2)^{\frac{\ka}{2}}$, with $\ka=1-\frac{1}{2(n-1)}$, establishes the $L^\infty$ estimate \eqref{Decay-es-ymh-2'} for $|\n_A u|$.
\end{proof}

\medskip
\section{Removable singularities for YMH fields in higher dimensions}\label{s: remov-singu}
In this section, we use the decay estimates \eqref{Decay-es-ymh-1'} and \eqref{Decay-es-ymh-2'} from Section \ref{s: decay-es} to prove the removable singularity theorem, Theorem \ref{remov-singu}. We first recall the local existence of a Coulomb gauge established by Smith and Uhlenbeck \cite{SU19}.
\begin{prop}\label{coulomb-gauge}
Let $B_1$ be the unit ball in $\Real^n$, where $n\geq 3$, and let $\mathcal{S}\subset B_1$ be a closed set of finite $(n-3)$-dimensional Hausdorff measure. There exists $\ep>0$ such that, for any connection $A$ smooth on $B_1\setminus\mathcal{S}$ whose curvature $F_A$ satisfies
\[\int_{B_1}|F_A|^pd\mu\leq \ep^p,\]
for some $p>n$, there exists a continuous gauge transformation $s$ such that $\bar{A}=s^*A\in W^{1,p}_{\mathrm{loc}}(B_1)$ and $\bar{A}$ has the following properties:
\begin{itemize}
\item[$(1)$] In $B_1$, $\bar{A}$ solves 
\begin{align}
 d^*\bar{A}=0, \label{eq-A} 
\end{align}

\item[$(2)$] There exists a constant $C$, depending only on $n$ and $p$, such that
\begin{align}
\norm{\bar{A}}_{W^{1,p}(B_{1/2})}\leq C\norm{F_{\bar{A}}}_{L^p}\leq C\ep.   \label{int-es-A} 
\end{align}
\end{itemize} 
\end{prop}

We now use Proposition \ref{coulomb-gauge} to prove Theorem \ref{remov-singu}.
\begin{proof}[Proof of Theorem \ref{remov-singu}]

The proof is divided into two steps.

\medskip
\noindent\emph{Step 1: In a local Coulomb gauge, the YMH field $(A,u)$ is smooth near the origin.}

\medskip
Applying decay estimate \eqref{key-es-1} from Theorem \ref{main-thm}, with both $\ep_0$ and $r_0$ sufficiently small, gives
\begin{align}
\int_{B_{r_0}}|F_A|^pdx\leq \ep^pr_0^{-2p+n}\label{int-es-F}
\end{align}
for $p>2n$, where $\ep$ is the constant in Proposition \ref{coulomb-gauge}. Setting $\tilde{A}(x)=r_0A(r_0x)$, we obtain
\[\int_{B_1}|F_{\tilde{A}}|^pdx\leq \ep^p.\]
Proposition \ref{coulomb-gauge} then provides a gauge $s$ on $B_1$ in which \eqref{eq-A} and \eqref{int-es-A} hold. After rescaling, these estimates become
\[d^*A_1=0 \quad\text{in}\,\, B_{r_0}\]
and
\begin{equation}\label{es-A}
\norm{A_1}_{W^{1,p}(B_{r_0/2})}\leq C(r_0,p)\ep.
\end{equation}
Here $A_1=s^*A=s^{-1}ds+s^{-1}As$.
	
Moreover, decay estimate \eqref{key-es-2} for $\n_Au$ from Theorem \ref{main-thm} implies that
\begin{equation}\label{es-u}
\norm{u_1}_{W^{1,\infty}(B_{r_0/2})}\leq C(r_0),
\end{equation}
where we set $u_1=s^*u$.
	
Under the gauge transformation $s$, the YMH field $(A_1,u_1)$ satisfies the extrinsic equation \eqref{eq-ymh2} on $B_{r_0}$:
\begin{equation*}
\begin{cases}
\De A_1+A_1\#dA_1+A_1\#A_1\#A_1=-u_1^{*}(\n_{A_1}u_1),\\
\De u_1+2A_1\#du_1+A_1(A_1u_1)=\Ga(u_1)(\n_{A_1} u_1,\n_{A_1}u_1)+2\n V(u_1).
\end{cases}
\end{equation*}
where $\Ga(u_1)$ is the second fundamental form of $N$ in $\Real^K$. Estimates \eqref{es-A}-\eqref{es-u} imply that
\[(\De A_1, \De u_1)\in L^\frac{p}{2}_{loc}(B_{r_0/2}),\]
and the $L^p$ estimates for elliptic equations imply that
\[(A_1,u_1)\in W^{2,\frac{p}{2}}_{loc}(B_{\rho_0})\hookrightarrow C^{1,1-\frac{2n}{p}}_{loc}(B_{r_0/2}).\]
	
A standard bootstrap argument applied to this $C^{1,1-\frac{2n}{p}}$ estimate yields $(A_1,u_1)\in C^\infty(B_{r_0/2})$.

\medskip
\noindent\emph{Step 2: A gluing argument produces a global continuous gauge on $B_1$ in which the YMH field $(A,u)$ is smooth.}

\medskip
Without loss of generality, assume that the connection $A$ has bounded $L^p$ norm on $B_{r_0/2}$, where $p>n$. Indeed, a radial gauge over $B^*_1$ is obtained by parallel transport along each ray $\ga_\th:=\{(r,\th):\ 0<r<1\}$ in $B_1^*$. In this gauge, $A=A_\th d\th$, and hence $r|A|=|A_\th|$, which implies
\[r|A|(r,\th)\leq \int_{0}^rt|F_A|(t,\th)dt.\]
Applying H\"older's inequality and estimate \eqref{int-es-F}, with $p>n$, gives (see also \cite{SU19})
\begin{align}
\norm{A}_{L^p(B_{r_0/2})}\leq Cr_0\norm{F_A}_{L^p(B_{r_0})}\leq C(r_0)\ep.\label{es-A-1}
\end{align}

Since $A_1$ and $A$ are smooth on $B^*_{r_0/2}$, it follows from
\begin{equation}\label{eq-s}
ds=sA_1-As
\end{equation}
that $s$ is smooth on $B^*_{r_0/2}$.

Let $\rho\leq r_0/4$ be a positive constant to be determined. For $q\in B^*_{\rho}$, replacing $s$ with $s^{-1}(q)s$ allows us to assume that $s(q)=\mathrm{id}$. Using estimates \eqref{es-A} and \eqref{es-A-1}, with $p>n$, in equation \eqref{eq-s}, we obtain
\begin{align*}
\sup_{y\in B^*_{\rho}}|s(y)-\mathrm{id}|\leq& C(r_0)\rho^{1-\frac{n}{p}}\norm{ds}_{L^p(B^*_{r_0/2})}\leq C(r_0)\ep\rho^{1-\frac{n}{p}}.
\end{align*}
Thus, by choosing $\rho$ small enough, we can rewrite $s=\exp\xi(x)$ on $B^*_\rho$, where $\xi: B^*_{\rho}\to \g$ is a smooth function such that $\xi(q)=0$.

Let $\eta\in C^\infty_0(B_1)$ be a smooth cutoff function such that $\eta(x)=1$ for $r\leq \frac{1}{4}\rho$ and $\eta(x)=0$ for $r\geq \frac{1}{2}\rho$. Define a global gauge on $B^*_1$ by
\begin{equation*}
\tilde{s}(x)=
\left\{
\begin{aligned}
& s\,\, &0<r\leq \frac{\rho}{4},\\
& \exp(\eta(x)\xi(x))\,\,&\frac{\rho}{4}\leq r\leq \frac{\rho}{2},\\
& \mathrm{id} \,\, &\frac{\rho}{2}\leq r<1.
\end{aligned}\right.
\end{equation*}
It follows that the YMH field $(\tilde{s}^*A,\tilde{s}^*u)$ is smooth on $B_1$.

Thus, the singularity of the YMH field $(A,u)$ is removable, and the bundle extends over the ball $B_1$.
\end{proof}

Finally, we prove Corollary \ref{remov-singu1}.

\begin{proof}
By Theorem \ref{remov-singu}, it suffices to show that for every $\ep_0>0$, there exists $R_0\leq 1$ such that
\[\sup_{B_{\rho(y)}\subset B_{R_0}}\(\frac{1}{\rho^{n-2}}\int_{B_{\rho}(y)}|\n_A u|^2+|F_A|dx\)\leq \ep^2_0,\]
which is precisely the energy bound \eqref{inq3'}.
	
For any $B_\rho(x)\subset B_1$, we have
\begin{align}
\frac{1}{\rho^{n-2}}\int_{B_\rho(x)}|\n_A u|^2+|F_A|dx\leq C(\int_{B_{\rho}}|F_A|^{\frac{n}{2}}dx+\int_{B_{\rho}}|\n_A u|^ndx)^{\frac{2}{n}},\label{es-1}
\end{align}
where $C>0$ depends only on $n$.

By \eqref{cor-con-enery}, there exists $0<R_0\leq 1$ such that
\begin{align}
\int_{B_{R_0}}|F_A|^{\frac{n}{2}}dx+\int_{B_{R_0}}|\n_A u|^ndx\leq \frac{1}{C^{\frac{n}{2}}}\ep^n_0.  \label{es-2}  
\end{align}
Combining \eqref{es-1} and \eqref{es-2} yields the desired energy bound \eqref{inq3'}.

The conclusion now follows from Theorem \ref{remov-singu}.
\end{proof}

\medskip
\noindent {\it\bf{Acknowledgments}}: B. Chen is partially supported by the National Natural Science Foundation of China (Grant No. 12301074) and the Guangzhou Basic and Applied Basic Research Foundation (Grant No. 2024A04J3637).

\medskip
\section*{Statements and declarations}
\noindent {\it\bf{Conflict of interest}}: The author declares no conflict of interest.

\medskip
\noindent {\it\bf{Data availability}}: Data sharing is not applicable to this article because no datasets were generated or analyzed during the study.

\end{document}